\documentclass[reqno]{amsart}

\title{Enumeration of sets of equiangular lines with common angle $\arccos(1/3)$}
\author[]{Kiyoto Yoshino$^\star$}
\thanks{$^\star$K.~Yoshino is supported by JSPS KAKENHI Grant Number JP21J14427.}
\address{Department of Computer Science, Faculty of Applied Information Science, Hiroshima Institute of Technology, Saeki Ward, Hiroshima, 731-5143, Japan}
\email{k.yoshino.n9@cc.it-hiroshima.ac.jp}

\usepackage[dvipdfmx]{graphicx}
\usepackage{amsmath,amsthm,amsfonts,amssymb,enumerate,mathrsfs,color,tikz}
\usetikzlibrary{calc}
\usepackage{mathtools}
\mathtoolsset{showonlyrefs,showmanualtags}
\usepackage{times,comment}
\usepackage{tikz,xcolor}
\usepackage[shortlabels]{enumitem}
\setlist[enumerate,1]{label={\upshape(\roman*)}}

\usepackage{geometry}
\numberwithin{equation}{section}

\newtheorem{lemma}{Lemma}[section]
\newtheorem{theorem}[lemma]{Theorem}

\newtheorem{corollary}[lemma]{Corollary}
\theoremstyle{definition}
\newtheorem{definition}[lemma]{Definition}
\newtheorem{remark}[lemma]{Remark}

\newtheorem{example}[lemma]{Example}

\newcommand{\R}{\mathbb{R}}

\newcommand{\Z}{\mathbb{Z}}		
\newcommand{\sA}{\mathsf{A}}
\newcommand{\sD}{\mathsf{D}}
\newcommand{\sE}{\mathsf{E}}
\newcommand{\sM}{\mathsf{M}}
\newcommand{\sL}{\mathsf{L}}

\newcommand{\cE}{\mathcal{E}}
\newcommand{\cL}{\mathcal{L}}
\newcommand{\cS}{\mathcal{S}}
\newcommand{\cT}{\mathcal{T}}

\newcommand{\cV}{\mathcal{V}}

\newcommand{\bu}{\mathbf{u}}
\newcommand{\be}{\mathsf{e}}
\newcommand{\bv}{\mathbf{v}}
\newcommand{\br}{\mathbf{r}}
\newcommand{\bx}{\mathbf{x}}

\newcommand{\bw}{\mathbf{w}}

\newcommand{\bj}{\mathbf{j}}

\newcommand{\tB}{C}
\newcommand{\oC}{P}
\newcommand{\tC}{Q}
\newcommand{\tF}{F}

\DeclareMathOperator{\rank}{rank}

\DeclareMathOperator{\Aut}{Aut}

\DeclareMathOperator{\Hom}{Hom}

\DeclareMathOperator{\sw}{sw}

\begin{document}

\keywords{Equiangular lines, Seidel matrices, Root lattices}
\subjclass[2020]{05C50}

\maketitle

\begin{abstract}
	In 2018, Sz\"{o}ll\H{o}si and \"{O}sterg\r{a}rd used a computer to enumerate sets of equiangular lines with common angle $\arccos(1/3)$ in dimension $7$.
	They observed that the numbers $\omega(n)$ of sets of $n$ equiangular lines with common angle $\arccos(1/3)$ in dimension $7$ are almost symmetric around $n=14$.
	In this paper, we prove without a computer that the numbers $\omega(n)$ are indeed almost symmetric by considering isometries from root lattices of rank at most $8$ to the root lattice $\sE_8$ of rank $8$ and type $E$.
	Also, they determined the number $s(n)$ of sets of $n$ equiangular lines with common angle $\arccos(1/3)$ for $n \leq 13$.
	We construct all the sets of equiangular lines with common angle $\arccos(1/3)$ in dimension greater than $7$ from root lattices of type $A$ or $D$ with the aid of switching roots.
	As an application, we determine the number $s(n)$ for every positive integer $n$.
\end{abstract}


\section{Introduction}
	A set of lines through the origin in a Euclidean space is \emph{equiangular} if any pair from these lines forms the same angle.
	Let $N(d)$ be the maximum cardinality of a set of equiangular lines in dimension $d$.
	The problem of determining the values of $N(d)$ has been under consideration since 1948~\cite{Haantjes1948}.	
	The values or bounds of $N(d)$ are known for $d \leq 43$~\cite{Barg2017, Greaves2022, Greaves2021, Greaves2020, Haantjes1948,  lemmens1973, Lint1966} as in Table~\ref{table:N(d)}.
	\begin{table}[hbtp]	\label{table:N(d)}
		\caption{The values or bounds of $N(d)$ for $d \leq 43$.}
		\begin{tabular}{|c|ccccccccccccccccc|}
			\hline
			$d$ & 2 & 3 & 4 & 5 & 6 & 7--14 & 15 & 16 & 17 & 18 & 19 & 20 & 21 & 22 & 23--41 & 42 & 43 \\
			\hline
			$N(d)$ & 3 & 6 & 6 & 10 & 16 & 28 & 36 & 40 & 48 & 57--59 & 72--74 & 90--94 & 126 & 176 & 276 & 276--288 & 344 \\
			\hline
		\end{tabular}
	\end{table}
	A general upper bound $N(d) \leq d(d+1)/2$ was proved by Lemmens and Seidel~\cite[Theorem~3.5]{lemmens1973}.
	Also, a general lower bound $\left(32 d^2+328 d+29\right) / 1089$ was proved by Greaves~et~al.~\cite[Corollary~2.8]{GHAF2016}.
	
	To determine the values of $N(d)$, sets of equiangular lines with a fixed angle have been considered.
	Let $N_\alpha(d)$ be the maximum cardinality of a set of equiangular lines with angle $\arccos(\alpha)$ in dimension $d$.
	In 1966, van Lint and Seidel proved the so-called relative bound $N_{\alpha}(d) \leq d(1-\alpha^2)/(1-d\alpha^2)$ for $d < 1/\alpha^2$~\cite[p.~342]{Lint1966}.
	Since $1/\alpha$ is an odd integer if $N_\alpha(d) > 2d$~\cite[Theorem~3.4]{lemmens1973},
	the values of $N_\alpha(d)$ where $1/\alpha$ is an odd integer have been studied.
	In 1973, Lemmens and Seidel determined the values of $N_{1/3}(d)$ for all $d$~\cite[Theorem~3.6]{lemmens1973}.
	In addition, they posed the so-called Lemmens-Seidel conjecture, which claims
	$
			N_{1/5}(d) = \max\left\{  
				276,
				\left\lfloor (3d-3)/2 \right\rfloor 
			\right\}
	$
	for $d \geq 23$.
	This was shown by the works of Lemmens and Seidel~\cite{lemmens1973}, Cao~et~al.~\cite{cao2022} and Yoshino~\cite{Yoshino2022}.
	In 2021, Jiang~et~al.\ proved for every integer $k \geq 2$, $N_{1/(2k-1)}(d) = \lfloor k(d-1)/(k-1) \rfloor$ for all sufficiently large $d$~\cite[Corollary~1.3]{Zhao2021}.

	In this paper, we investigate the enumeration and structure of sets of equiangular lines with common angle $\arccos(1/3)$.
	For a set $U$ of equiangular lines, we denote by $\langle U \rangle_{\R}$ the smallest $\R$-linear subspace containing $U$.
	We say two sets $U$ and $U'$ of equiangular lines are \emph{equivalent} if there exists an $\R$-linear mapping $f$ from $\langle U \rangle_{\R}$ to $\langle U' \rangle_{\R}$ preserving inner products such that $f(U) = U'$.
	Throughout this paper, we count the number of sets of equiangular lines up to equivalence.
	In 2018, Sz\"{o}ll\H{o}si and \"{O}sterg\aa rd~\cite{Szollosi2018} with a computer enumerated the sets of $n$ equiangular lines with common angle $\arccos(1/3)$ for $n \leq 13$.
	As a result, they determined $s(n)$ and $t(n)$ for $n \leq 13$ as in Table~\ref{table:Szo2}.
	Here, $s(n)$ is the number of sets of $n$ equiangular lines with common angle $\arccos(1/3)$,
	and $t(n)$ is the number of sets $U$ of $n$ equiangular lines with common angle $\arccos (1/3)$ and $\dim \langle U \rangle_{\R} < n$.
	\begin{table}[hbtp] 
		\centering
		\caption{The numbers $s(n)$ and $t(n)$} \label{table:Szo2}
		\begin{tabular}{|c|cccccccccccccc|}
			\hline
			$n$ & 0 & 1 & 2 & 3 & 4 & 5 & 6 & 7 & 8 & 9 & 10 & 11 & 12 & 13 \\ 
			\hline \hline
			$s(n)$ & 1 & 1 & 1 &2 & 3 & 5 & 9 & 16 & 25 & 40 & 58 & 75 & 96 & 108 \\
			\hline
			$t(n)$ &0 & 0 & 0 & 0 & 1 & 1 & 4 & 9 & 23 & 38 & 56 & 73 & 94 & 106 \\
			\hline
		\end{tabular}
	\end{table}				
	Additionally, they enumerated the sets of $n$ equiangular lines with common angle $\arccos(1/3)$ in dimension $7$ for every $n \leq N_{1/3}(7) = 28$ with a computer,
	and determined the number $\omega(n)$ of them as in Table~\ref{table:Szo}.
	\begin{table}[hbtp] 
		\centering
		\caption{The numbers $\omega(n)$ of sets of $n$ equiangular lines with common angle $\arccos(1/3)$ in dimension $7$} \label{table:Szo}
		\begin{tabular}{|c|ccccccccccccccc|}
			\hline
			$n$ & 0 & 1 & 2 & 3 & 4 & 5 & 6 & 7 & 8 & 9 & 10 & 11 & 12 & 13 & 14 \\ 
			\hline
			$\omega(n)$ & 1 & 1 & 1 & 2 & 3 & 5 & 9 & 16 & 23 & 37 & 54 & 70 & 90 & 101 & 103 \\
			\hline \hline
			$n$ &15 & 16 & 17 & 18 & 19 & 20 & 21 & 22 & 23 & 24 & 25 & 26 & 27 & 28 & 29 \\
			\hline
			$\omega(n)$ & 101 & 90 & 70 & 54 & 37 & 23 & 16 & 10 & 5 & 3 & 2 & 1 & 1 & 1 & 0 \\
			\hline
		\end{tabular}
	\end{table}	
	Then, they noticed that the numbers $\omega(n)$ are symmetric around $n=14$ except for $n=6, 22$.
	Note that we treat the empty set as the set of equiangular lines of cardinality $0$.
	In 2020, Lin and Yu~\cite{Lin2020} employed a computational approach to determine the maximal sets of equiangular lines with common angle $\arccos(1/3)$ exactly in dimension $8$.
	Subsequently, Cao~et~al.~\cite{cao2021} introduced the concept of switching roots, and classified the maximal sets of equiangular lines with common angle $\arccos(1/3)$ by the containment relation among root lattices.
		
	In this study, we use switching roots as in~\cite{cao2021}, and apply more precise results on root lattices.
	In the first half, we consider the orbits of some sets of roots of a root lattice under the left action of the Weyl group.
	As a result, we obtain the first main result, Theorem~\ref{thm:Sn}, which enumerates the sets of equiangular lines with common angle $\arccos(1/3)$ in dimension greater than $7$, and have the following corollary.
	\begin{corollary}	\label{cor:Sn}
		For an integer $n$, $s(n) =  t(n)+2$ if $n \geq 8$, and
		\begin{align*}
			s(n)
			= \begin{cases}
				\omega(n) & \text{ if } 0 \leq n \leq 7, \\
				\omega(n) + n - 6  & \text{ if } 8 \leq n \leq 12, \\
				\omega(n) + \left\lfloor \frac{n}{2} \right\rfloor + 1 & \text{ if } 13 \leq n \leq 28, \\
				\left\lfloor \frac{n}{2} \right\rfloor + 1 & \text{ if } n \geq 29.
			\end{cases}
		\end{align*}
	\end{corollary}
	In the second half, we prove the second main result, Theorem~\ref{thm:sym}, which shows that the orbits of some sets of roots of $\sE_8$ correspond almost one-to-one with the sets of equiangular lines with common angle $\arccos(1/3)$ in dimension $7$.
	Here, $\sE_8$ is the root lattice of rank $8$ and type $E$.
	Furthermore, we see that $\omega(n)$ is almost symmetric around $n=14$ as in Corollary~\ref{cor:sym} without a computer.
	\begin{corollary}	\label{cor:sym}
		$\omega(n) = \omega(28-n)$ for $n \in \{0,\ldots,28\} \setminus \{6,22\}$, and $\omega(6)+1=\omega(22)$.
	\end{corollary}	
	This paper is organized as follows.
	In Section~\ref{sec:notations}, we introduce some notations for graphs and Seidel matrices,
	and explain the relation between Seidel matrices and sets of equiangular lines.
	In Section~\ref{sec:root}, we introduce basic definitions for root lattices.
	In Section~\ref{sec:sr}, we present some known results on switching roots.
	In Section~\ref{sec:3}, we prepare lemmas to prove the main results.
	In Section~\ref{sec:m1}, we establish the first main result, Theorem~\ref{thm:Sn}, and obtain Corollary~\ref{cor:Sn}.
	In Section~\ref{sec:embeddings}, we consider the isometries from a root lattice to $\sE_8$ in order to prove the second main result.
	In Section~\ref{sec:relations}, we establish the second main result, Theorem~\ref{thm:sym}, and show Corollary~\ref{cor:sym}.
	
\section{Notations}	\label{sec:notations}
	Throughout this paper, we consider undirected graphs without loops and multiedges.
	Let $G$ be a graph.
	Denote by $V(G)$ the set of vertices, and by $E(G)$ the set of edges.
	For two vertices $x$ and $y$ of $G$, we write $x \sim y$ if $\{x,y\} \in E(G)$.
	Denote by $A(G)$ the adjacency matrix of $G$.
	Denote by $\bj$ the all-one vector, and by $\be_i$ the vector of which the $i$-th entry is $1$ and the others are $0$.
	Denote by $I$ and $J$ the identity matrix and the all-one matrix, respectively.

	\begin{definition}	
		A \emph{Seidel matrix} is a symmetric matrix with zero diagonal and all off-diagonal entries $\pm1$.
		For two Seidel matrices $S$ and $S'$, write $S {\sim_{\sw}} S'$ and say that $S$ and $S'$ are \emph{switching equivalent} if there exists a signed permutation matrix $P$ such that $P^\top S P = S'$.
		In addition, write $[S]$ for the set of Seidel matrices switching equivalent to a Seidel matrix $S$,
		and call it the \emph{switching class} of $S$.
		For a graph $G$, define $S(G) := J-I-2A(G)$.
		Two graphs $G$ and $H$ are said to be \emph{switching equivalent} if $S(G)$ and $S(H)$ are switching equivalent.
		Write $[G]$ for the set of graphs switching equivalent to a graph $G$.
	\end{definition}

	Note that for any Seidel matrix $S$, there exists a graph $G$ such that $S=S(G)$.	
	Also, note that if $G$ and $H$ are switching equivalent, then $H$ is isomorphic to $G^U$ for some $U \subset V(G)$.
	Here, the graph $G^U$ is defined as $V=V(G^U):=V(G)$ and 
	\begin{align*}
		x\sim y \ \text{in} \ G^U  
		:\iff
		 \begin{cases}
			x\sim y \text{ in } G &\text{ if }  \{x,y\} \subset U\text{ or } \{x, y\} \subset V\setminus U, \\
			x\not\sim y \text{ in } G &\text{ if } \{x,y\} \not\subset U\text{ and } \{x, y\} \not\subset V\setminus U \\
		 \end{cases}
	\end{align*}
	for $x \neq y \in V$.

	We explain the relation between sets of equiangular lines and Seidel matrices.
	Fix a set $U$ of $n$ equiangular lines with common angle $\arccos(\alpha)$ in dimension $d$.
	We may take unit vectors $\bu_1,\ldots,\bu_n \in \R^d$ such that $U = \{ \R \bu_1, \ldots, \R \bu_n \}$.
	There exists a Seidel matrix $S$ such that
	$
		I -  \alpha S = B^\top B,
	$
	where $ B = [\bu_1,\ldots,\bu_n]$, that is, $I -  \alpha S$ is the Gram matrix of $\bu_1,\ldots,\bu_n$.
	Consequently, the set $U$ of equiangular lines induces the Seidel matrix $S$ with largest eigenvalue at most $1/\alpha$ and $\rank(I - \alpha S) \leq d$ up to switching.
	In addition, it induces the graph $G$ with $S=S(G)$ up to switching.
	Conversely, we can recover $U$ from the Seidel matrix $S$ or the graph $G$.
	\begin{definition}
		Let $n$ be a non-negative integer.
		Let $\cS_{n}$ be the set of Seidel matrices of order $n$ with largest eigenvalue at most $3$.
		Let $\cT_{n}$ be the set of Seidel matrices of order $n$ with largest eigenvalue equal to $3$.
		Let $\Omega_{n}$ be the subset of Seidel matrices $S \in \cS_{n}$ with $\rank(3I-S) \leq 7$,
		and set $\Theta_{n} := \Omega_{n} \cap \cT_{n}$.
	\end{definition}
	We see that $s(n) = |\cS_{n}/\sim_{\sw}|$, $t(n) = |\cT_{n}/\sim_{\sw}|$ and $\omega(n) = |\Omega_{n}/\sim_{\sw}|$.

\section{Root lattices} \label{sec:root}
    In this section, we introduce basic definitions for lattices, and root lattices.
    For a reference, see~\cite{Lattices2013}.
    \begin{definition}
        A \emph{lattice} is the set $\sL$ of integral linear combinations of some $\R$-linearly independent vectors $\bu_1,\ldots,\bu_m  \in \R^n$ ($m \leq n$),
        that is, $\sL = \Z \bu_1 \oplus \cdots \oplus \Z \bu_m.$
        In this case, $m$ is called the \emph{rank} of the lattice $\sL$.
		For vectors $\bv_1, \ldots, \bv_k$ which are not necessarily $\R$-linearly independent, if a lattice is expressed as the set of integral linear combinations of $\bv_1, \ldots, \bv_k$, then this lattice is denoted by $\langle \bv_1, \ldots, \bv_k \rangle$.
        Also, denote by $(\bu,\bv)$ the inner product of two vectors $\bu$ and $\bv$.
        A lattice $\sL$ is said to be \emph{integral} if $(\bu,\bv) \in \Z$ for any $\bu,\bv \in \sL$.
        For each vector $\bu$, we call $(\bu,\bu)$ the \emph{norm} of $\bu$.
    \end{definition}
		
	We often consider vectors $\bu_1, \ldots, \bu_n$ whose Gram matrix is a positive semidefinite matrix with integer entries.
	It is known that any discrete $\Z$-submodule of $\R^n$ forms a lattice. For details, see~\cite{Lattices2013}.
	Since the norm of any integral linear combination of $\bu_1, \ldots, \bu_n$ is an integer, the set of integral linear combinations of $\bu_1, \ldots, \bu_n$ forms an integral lattice.
	Thus, this lattice can be written as $\langle \bu_1, \ldots, \bu_n \rangle$.

    \begin{definition}
	    A $\Z$-linear injection $f : \sL \to \sM$ between two lattices $\sL$ and $\sM$ is called an \emph{isometry} 
	    if $(f(\bu),f(\bv)) = (\bu,\bv)$ for any $\bu,\bv \in \sL$.
	    Two lattices $\sL$ and $\sM$ are said to be \emph{isometric} if there is a bijective isometry between them.
	    A bijective isometry from a lattice $\sL$ to itself is called an \emph{automorphism} of $\sL$.
	    The \emph{automorphism group}, denoted by $\Aut(\sL)$, of a lattice $\sL$ is defined to be the set of automorphisms of $\sL$.
	\end{definition}

    A lattice $\sL$ is said to be \emph{reducible} if it is the orthogonal direct sum of two nonzero lattices. Otherwise it is said to be \emph{irreducible}.
    A vector of norm $2$ in a lattice is called a \emph{root}, and an integral lattice generated by roots is called a \emph{root lattice}.
    It is known that the irreducible root lattices are enumerated up to bijective isometry as follows (cf.~\cite[Theorem~1.2]{Lattices2013}).
    \begin{align*}
        \sA_n &:= \{ \bv \in \Z^{n+1}  :  (\bv,\bj) = 0\} \quad (n \in \Z_{\geq 1}),\\
        \sD_n &:= \{ \bv \in \Z^n  :  (\bv,\bj) \in 2\Z \} \quad (n \in \Z_{\geq 4}),\\
        \sE_8 &:= \sD_8 \sqcup \left( \bj/2 + \sD_8 \right),\\
        \sE_7 &:= \{ \bv \in \sE_8  :  (\bv,\be_1-\be_2)=0\}, \\
        \sE_6 &:= \{ \bv \in \sE_8  :  (\bv,\be_1-\be_2) = (\bv,\be_2-\be_3) = 0\} .
    \end{align*}

	\begin{definition}
		Let 
		$\cL := \{\sA_n : n \geq 1\} \cup \{\sD_n : n \geq 4 \} \cup \{\sE_6,\sE_7,\sE_8\}$
		and 
		$$\cE := \{\sA_n : n=1,\ldots,8\} \cup \{\sD_n : n=4,\ldots,8 \} \cup \{\sE_6,\sE_7,\sE_8\}.$$
	\end{definition}

    The \emph{Weyl group} $W(\sL)$ of a root lattice $\sL$ is the subgroup of $\Aut(\sL)$ generated by all the reflections $s_\br$ corresponding to roots $\br$, where $s_{\br}(\bx) = \bx - 2\br(\bx,\br)/(\br,\br) = \bx-\br(\bx,\br)$.
    In this paper, we frequently make use of the following lemma concerning Weyl groups.
    \begin{lemma}[{\cite[Lemma~1.10]{Lattices2013}}]	\label{lem:transitive}
        The Weyl group of an irreducible root lattice acts transitively on the set of its roots.
    \end{lemma}

\section{Switching roots}	\label{sec:sr}
	Cao~et~al.~\cite{cao2021} introduced the concept of switching roots and a novel method for constructing lattices from Seidel matrices.
	In this section, we describe some results related to switching roots.
	\begin{definition}
		The \emph{cone} $\hat{G}$ over a graph $G$ is defined as the graph obtained by adding a new vertex to $G$ and connecting it to all the vertices of $G$.
		Here, we call the newly added vertex the \emph{apex} for $G$.
	\end{definition}
	\begin{theorem}[{\cite[Corollary~2.4]{cao2021}}]	\label{thm:Cao}
		For a graph $G$, the following are equivalent.
		\begin{enumerate}
			\item The Seidel matrix $S(G)$ has largest eigenvalue at most $3$.	\label{thm:Cao:1}	
			\item The adjacency matrix $A(\hat{G})$ of the cone $\hat{G}$ over $G$ has smallest eigenvalue at least $-2$. 	\label{thm:Cao:2}
		\end{enumerate}
		Furthermore, if \ref{thm:Cao:1} or \ref{thm:Cao:2} holds, then $\rank(3I-S(G))+1 = \rank(A(\hat{G})+2I).$
	\end{theorem}	
	In this theorem, if~\ref{thm:Cao:2} holds, then the graph $G$ satisfies that $A(\hat{G})+2I = \tB^\top \tB$ for some vectors $\bu_1,\ldots,\bu_n, \br$, where $\tB =  [\bu_1,\ldots,\bu_n, \br].$
	Here, $\br$ corresponds to the apex for $G$.
	This root $\br$ is called a \emph{switching root}~\cite{cao2021}.
		
	\begin{definition}
		Define $[\bu]_{\br} := \{ \bu, \br - \bu\}$ for two roots $\br$ and $\bu$.
		Let $\sL$ be an irreducible root lattice.
		Assume $\br$ is a root in $\sL$.
		Let $N_{\br}(\sL)$ be the set of roots $\bu \in \sL$ with $(\bu,\br)=1$.
		Define the equivalence relation $\sim_{\br}$ on the set $N_{\br}(\sL)$ such that for $\bu,\bv \in N_{\br}(\sL)$, $\bu {\sim_\br} \bv$ if and only if $\bu=\bv$ or $\bu = \br-\bv$.
	\end{definition}
	Note that for an irreducible root lattice $\sL$ and a root $\br \in \sL$, the set $[\bu]_\br$ is the equivalence class of $\bu \in N_{\br}(\sL)$ under ${\sim_\br}$.
	For a set $A$ and a non-negative integer $n$, denote by $\binom{A}{n}$ the set of $n$-subsets of $A$.    
	\begin{remark}    
		Let $\sL$ be an irreducible root lattice, and $\br$ a root.
	    	We remark that for any roots $\bu, \bv \in N_{\br}(\sL)$,
		\begin{align*}
			(\bu,\bv) = \begin{cases}
				2 & \text{ if } \bu = \bv,\\
				-1 & \text{ if } \bu = \br - \bv, \\
				0,1 & \text{ otherwise}.
			\end{cases}
	    	\end{align*}
		Hence, if
		\begin{align*}
			\{ [\bu_1]_\br ,\ldots, [\bu_n]_\br \} \in \binom{N_{\br}(\sL)/\sim_\br}{n},
		\end{align*}
 		then there exists a graph $G$ satisfying $A(G)+2I = B^\top B$ and $A(\hat{G})+2I = \tB^\top \tB$, where $B = [\bu_1,\ldots,\bu_n]$ and  $\tB = [\bu_1,\ldots,\bu_n, \br]$.
	\end{remark}

	The following is necessary to associate a switching class with a lattice.
	Lemma~\ref{lem:vectors-to-graphs}~\ref{lem:vectors-to-graphs:1} has been essentially claimed after~\cite[Definition~2.1]{cao2021}.
	Additionally, Lemma~\ref{lem:vectors-to-graphs}~\ref{lem:vectors-to-graphs:2} holds by definition.
	\begin{lemma}	\label{lem:vectors-to-graphs}
		Let $G$ be a graph of order $n$ satisfying $A(\hat{G})+2I = \tB^\top \tB$ for some vectors $\bu_1,\ldots,\bu_n,\br$, where $\tB = [\bu_1, \ldots, \bu_n, \br]$.
		Here, the vector $\br$ corresponds to the apex for $G$.
		\begin{enumerate}
			\item We have the surjection	\label{lem:vectors-to-graphs:1}
			\begin{align*}
				\begin{array}{ccc}
					[\bu_1]_\br \times \cdots \times [\bu_n]_\br & \to & [G] \\
					(\bv_1,\ldots,\bv_n) & \mapsto & H
				\end{array}
			\end{align*}
			where $A(H)+2I = B'^\top B'$ with $B' = [\bv_1,\ldots,\bv_n]$.
			\item \label{lem:vectors-to-graphs:2}
			For any $(\bv_1,\ldots,\bv_n) \in [\bu_1]_\br \times \cdots \times [\bu_n]_\br$,
			we have $\langle \bu_1,\ldots,\bu_n,\br \rangle = \langle \bv_1,\ldots,\bv_n,\br \rangle$.
		\end{enumerate}
	\end{lemma}

\section{Seidel matrices and root lattices}	\label{sec:3}
	In this section, we discuss various relations between Seidel matrices and root lattices.

\subsection{Decomposition of the sets $\cS_{n}$ into sets $\cS_{n}(\sL)$}
	We introduce the sets $\cS_{n}(\sL)$ coming from root lattices $\sL$, and decompose the sets $\cS_{n}$ into $\cS_{n}(\sL)$.
	We then further decompose the sets $\cS_{n}/\sim_{\sw}$ into $\cS_{n}(\sL)/\sim_{\sw}$.
	
	\begin{definition}	\label{dfn:cS(L)}
		Let $\sL$ be an irreducible root lattice, and $n$ a non-negative integer.
		Let $\cS_{n}(\sL)$ denote the set of Seidel matrices $S$ 
		such that there exists a graph $G$ with $S=S(G)$ satisfying $A(\hat{G})+2I = \tB^\top \tB$ for some generators $\bu_1,\ldots,\bu_n,\br$ of $\sL$, where $\tB = [\bu_1,\ldots,\bu_n,\br]$.
		Here, the vector $\br$ corresponds to the apex for $G$.
	\end{definition}	

	\begin{example}
		For example, $S(K_3) \in \cS_{3}(\sA_4)$, where $K_3$ is the complete graph of order $3$.
		Indeed, $\be_1-\be_2$, $\be_1-\be_3$, $\be_1-\be_4$, $\be_1-\be_5$ are generators of $\sA_4$.
		In addition, letting $\tB = [\be_1-\be_2$, $\be_1-\be_3$, $\be_1-\be_4, \be_1-\be_5]$, we have $A(\hat{K_3})+2I = A(K_4)+2I = \tB^\top \tB$.
	\end{example}

	We present lemmas to decompose the sets $\cS_n$ into sets $\cS_n(\sL)$.
	First, we prove the following lemma which shows that the root lattices appearing in the decomposition of $\cS_n$ are irreducible.
	\begin{lemma}	\label{lem:irr}
		Let $\sL$ be a root lattice. Assume that there exists a graph $G$ of order $n$ satisfying $A(\hat{G})+2I = \tB^\top \tB$ for some generators $\bu_1,\ldots, \bu_n, \br$ of $\sL$, where $\tB=[\bu_1,\ldots,\bu_n,\br]$.
		Then, $\sL$ is irreducible.
	\end{lemma}
	\begin{proof}
		Let $X := \{\bu_1,\ldots,\bu_n,\br\}$.
		Assume that there are two lattices $\sL_1$ and $\sL_2$ such that $\sL = \sL_1 \perp \sL_2$.
		Then, $X = (X \cap \sL_1) \sqcup (X \cap \sL_2)$.
		Since the cone $\hat{G}$ is connected, we see that either $X = X \cap \sL_1$ or $X  = X \cap \sL_2$.
		Hence, either $\sL = \sL_1$ or $\sL = \sL_2$.
		This means that $\sL$ is irreducible.
	\end{proof}

	Next, we present the following lemma, which shows that the sets $\cS_n(\sL)$ appearing in the decomposition of $\cS_n$ are pairwise disjoint.

	\begin{lemma}	\label{lem:SS}
		Let $n$ and $m$  be non-negative integers. 
		Let $\sL$ and $\sM$ be irreducible root lattices in $\cL$.
		If $\cS_{n}(\sL) \cap \cS_{m}(\sM) \neq \emptyset$, then $n=m$ and $\sL=\sM$.
	\end{lemma}
	\begin{proof}
		Assume that there exists $S \in \cS_{n}(\sL) \cap \cS_{m}(\sM)$.
		Then, $n=m$.
		By $S \in \cS_{n}(\sL)$, there exists a graph $G$ with $S=S(G)$ satisfying $A(\hat{G})+2I = \tB^\top \tB$ for some generators $\bu_1,\ldots,\bu_n,\bu_0$ of $\sL$, where $\tB = [\bu_1,\ldots,\bu_n,\bu_0]$.
		By $S \in \cS_{m}(\sM)$, there exists a graph $H$ with $S=S(H)$ satisfying 
		\begin{align}	\label{lem:SS:0}
			A(\hat{H})+2I = \tC^\top \tC
		\end{align}
		for some generators $\bw_1,\ldots,\bw_m,\bw_0$ of $\sM$, where $\tC = [\bw_1,\ldots,\bw_m,\bw_0]$. 
		Since $S(G) = S(H)$, the graphs $G$ and $H$ are switching equivalent.
		By Lemma~\ref{lem:vectors-to-graphs}~\ref{lem:vectors-to-graphs:1}, 
		there exists $$(\bv_1,\ldots, \bv_n) \in [\bu_1]_{\bu_0} \times \cdots \times [\bu_n]_{\bu_0}$$ such that $$A(H) + 2I = [\bv_1,\ldots,\bv_n]^\top [\bv_1,\ldots,\bv_n].$$
		Furthermore, we have
		\begin{align}
			A(\hat{H}) + 2I =  \tF^\top  \tF,
		\end{align}
		where $\tF = [\bv_1,\ldots,\bv_n,\bu_0]$.
		Hence, this together with~\eqref{lem:SS:0} implies
		\begin{align}
			\tC^\top \tC = \tF^\top  \tF.
		\end{align}
		This implies that $\langle \bw_1, \ldots,\bw_m ,\bw_0 \rangle$ and $\langle \bv_1,\ldots,\bv_n, \bu_0 \rangle$ are isometric.
		Also, by Lemma~\ref{lem:vectors-to-graphs}~\ref{lem:vectors-to-graphs:2},
		\begin{align}	\label{lem:SS:1}
			\langle \bu_1,\ldots,\bu_n, \bu_0 \rangle = \langle \bv_1,\ldots,\bv_n, \bu_0 \rangle.
		\end{align}
		Thus
		\begin{align}
			\sL = \langle \bu_1,\ldots,\bu_n, \bu_0 \rangle 
			= \langle \bv_1,\ldots,\bv_n, \bu_0 \rangle 
			\quad \text{ and } \quad
			\langle \bw_1, \ldots,\bw_m ,\bw_0 \rangle = \sM
		\end{align}
		are isometric.
		This is the desired conclusion.
	\end{proof}

	The following lemma essentially provides the decompositions of $\cS_n$ and $\cT_n$.
	
	\begin{lemma}	\label{lem:E}
		Let $n$ and $r$ be non-negative integers.
		Let $S$ be a Seidel matrix.
		The following are equivalent.
		\begin{enumerate}
			\item $S \in \cS_{n}$ and $\rank(3I-S)=r$.	\label{lem:E:1}
			\item $S \in \cS_{n}(\sL)$ for some $\sL \in \cL$ with $\rank(\sL)=r+1$.	\label{lem:E:2}
		\end{enumerate}
		Additionally, the following are equivalent.
		\begin{enumerate}
			\item[\textup{(iii)}] $S \in \cT_{n}$.	
			\item[\textup{(iv)}] $S \in \cS_{n}(\sL)$ for some $\sL \in \cL$ with $\rank(\sL) \leq n$.	
		\end{enumerate}
	\end{lemma}
	\begin{proof}
		First, we show that \ref{lem:E:1} and~\ref{lem:E:2} are equivalent.
		Assume~\ref{lem:E:1}.
		Let $G$ be a graph with $S=S(G)$.
		By Theorem~\ref{thm:Cao}, $A(\hat{G})$ has smallest eigenvalue at least $-2$ and satisfies $\rank(A(\hat{G})+2I)=r+1$.
		Hence, $A(\hat{G})+2I = \tB^\top \tB$ and $\rank \tB = r+1$ for some vectors $\bu_1,\ldots,\bu_n, \bu_0$, where $\tB = [\bu_1,\ldots,\bu_n,\bu_0]$.
		Let $\sL := \langle \bu_1,\ldots,\bu_n, \bu_0 \rangle$.
		Then, $\sL$ is an irreducible root lattice by Lemma~\ref{lem:irr}, and $\rank(\sL) = r+1$ follows from $\rank \tB = r+1$.
		Here, we may assume $\sL \in \cL$  (cf.~\cite[Theorem~1.2]{Lattices2013}).
		Therefore, \ref{lem:E:2} follows.
		
		Conversely, we assume~\ref{lem:E:2}.
		Then, there exists a graph $G$ with $S=S(G)$ satisfying $A(\hat{G})+2I = \tB^\top \tB$ for some generators $\bu_1,\ldots,\bu_n, \bu_0$ of $\sL$, where $\tB = [\bu_1,\ldots,\bu_n, \bu_0]$.
		Hence, $A(\hat{G})$ has smallest eigenvalue at least $-2$.
		Additionally, $\rank(A(\hat{G})+2I) = \rank \tB = \rank(\sL) = r+1$.
		Thus, by Theorem~\ref{thm:Cao}, $S(G)$ has largest eigenvalue at most $3$ and $\rank(3I-S(G)) = r$.
		Therefore, \ref{lem:E:1} follows.
		
		Next, we show that \textup{(iii)} and \textup{(iv)} are equivalent.
		Indeed, we see by the definition of $\cT_{n}$ that $S \in \cT_{n}$ if and only if $S \in \cS_{n}$ and $\rank(3I-S) \leq n-1$.
		Since \ref{lem:E:1} and~\ref{lem:E:2} are equivalent, we conclude that \textup{(iii)} and~\textup{(iv)} are equivalent.
		This is the desired conclusion.
	\end{proof}

	From the above lemmas, we explicitly decompose the sets $\cS_{n}$ and $\cT_n$ into sets $\cS_{n}(\sL)$.
	\begin{lemma}	\label{lem:Seidel decomp}
		Let $n$ be a non-negative integer.
		Then
		\begin{align}	\label{lem:Seidel decomp:3}
			\Omega_{n} = \bigsqcup_{\sL \in \cE} \cS_{n}(\sL)
			\quad \text{ and } \quad
			\Theta_{n} = \bigsqcup_{\sL \in \cE  \text{ with } \rank(\sL) \leq n} \cS_{n}(\sL).
		\end{align}
		Also,
			\begin{align}	\label{lem:Seidel decomp:1}
				\cS_{n} 
				= \bigsqcup_{\sL \in \cL} \cS_{n}(\sL)
				= \Omega_{n}  \sqcup \left( \bigsqcup_{\sL \in \cL \setminus \cE} \cS_{n}(\sL) \right)
			\end{align} 
		and
			\begin{align}	\label{lem:Seidel decomp:2}
			\cT_{n} = \bigsqcup_{\sL \in \cL  \text{ with } \rank(\sL) \leq n} \cS_{n}(\sL)
			= \Theta_{n}  \sqcup \left( \bigsqcup_{\sL \in \cL \setminus \cE \text{ with } \rank(\sL) \leq n} \cS_{n}(\sL) \right).
			\end{align}
	\end{lemma}	
	\begin{proof}
		This follows from Lemmas~\ref{lem:SS} and~\ref{lem:E}.
	\end{proof}
	
	Additionally, from the decomposition in Lemma~\ref{lem:Seidel decomp:3}, we obtain decompositions of $\cS_n/\sim_{\sw}$ and $\cT_n/\sim_{\sw}$.
	This lemma will be used to determine the values of $s(n)$, $t(n)$, and $\omega(n)$.

	\begin{lemma}	\label{lem:Seidel decomp2}
		For a non-negative integer $n$,
		\begin{align}	\label{lem:Seidel decomp2:1}
			\cS_{n}/{\sim_{\sw}} 
			= \Omega_{n}/{\sim_{\sw}}  \sqcup \left( \bigsqcup_{\sL \in \cL \setminus \cE} \cS_{n}(\sL)/{\sim_{\sw}} \right)
		\end{align}
		and
		\begin{align}	\label{lem:Seidel decomp2:2}
			\cT_{n}/{\sim_{\sw}} 
			= \Theta_{n}/{\sim_{\sw}}  \sqcup \left( \bigsqcup_{\sL \in \cL \setminus \cE \text{ with } \rank(\sL) \leq n} \cS_{n}(\sL)/{\sim_{\sw}} \right).
		\end{align}
	\end{lemma}
	\begin{proof} 
		By Lemma~\ref{lem:Seidel decomp}, it suffices to show that $S' \in  \cS_{n}(\sL)$ for any $S \in  \cS_{n}(\sL)$ and $S' \in [S]$.
		Indeed, we fix such Seidel matrices $S$ and $S'$.
		By Definition~\ref{dfn:cS(L)}, there exists a graph $G$ with $S=S(G)$ satisfying $A(\hat{G})+2I = \tB^\top \tB$ for some generators $\bu_1,\ldots,\bu_n,\br$ of $\sL$, where $\tB = [\bu_1,\ldots,\bu_n,\br]$.
		Here, the vector $\br$ corresponds to the apex for $G$.
		Let $H$ be a graph such that $S' = S(H)$.
		Then, $H \in [G]$ follows from the assumption $S' \in [S]$.
		By Lemma~\ref{lem:vectors-to-graphs}~\ref{lem:vectors-to-graphs:1}, there exists $(\bv_1,\ldots,\bv_n) \in [\bu_1]_\br \times \cdots \times [\bu_n]_{\br}$ such that $A(H)+2I = \oC^\top \oC$, where $\oC = [\bv_1,\ldots,\bv_n]$.
		Furthermore, $A(\hat{H})+2I = \tC^\top \tC$, where $\tC = [\bv_1,\ldots,\bv_n,\br]$.
		By Lemma~\ref{lem:vectors-to-graphs}~\ref{lem:vectors-to-graphs:2}, we have $\langle \bv_1,\ldots,\bv_n, \br \rangle = \langle \bu_1,\ldots,\bu_n, \br \rangle = \sL$.
		This means that $\bv_1,\ldots,\bv_n, \br$ are generators of $\sL$.
		Therefore, $S' = S(H) \in \cS_{n}(\sL)$.
	\end{proof}

	Next, we present an example of the decomposition~\eqref{lem:Seidel decomp2:1} given in Lemma~\ref{lem:Seidel decomp2}.
	\begin{example}	\label{ex:fig}
		All graphs of order at most $7$ are explicitly enumerated up to switching in~\cite[Tables~4.1 and~4.2]{Lint1966}.  
		Figures~\ref{fig:5} and~\ref{fig:6}, taken from~\cite[Table~4.1]{Lint1966}, show complete systems of representatives of $\cS_n/\sim_{\sw}$ of cardinality $s(n)$ for $n=2,\ldots,6$.
	\begin{figure}[ht]
	  \centering
	  \begin{tikzpicture}[scale=0.5]
	    \foreach \k/\n/\edges/\glabel in {
	      1/2/{}/$\sA_3$,
	      2/3/{2/3}/$\sA_4$,
	      3/3/{}/$\sD_4$,
	      4/4/{1/2,3/4}/$\sA_5$,
	      5/4/{3/4}/$\sD_5$,
	      6/4/{}/$\sD_4$
	    }{
	      \begin{scope}[xshift=(\k-1)*5cm]
	        \foreach \i in {1,...,\n}{
	          \node[circle,draw,fill=black,
	                inner sep=0.8pt,minimum size=3pt]
	                (v\k\i) at ({360/\n*(\i-1)}:1) {};
	        }
	        \foreach \u/\v in \edges {
	          \draw (v\k\u) -- (v\k\v);
	        }
	        \node[below] at (0,-1.4) {\glabel};
	      \end{scope}
	    }
	  \end{tikzpicture}
	 \\ \vspace{10pt}	
	 \begin{tikzpicture}[scale=0.5]
				\foreach \k/\edges/\glabel in {
					1/{1/2,2/5,5/1,3/4}/$\sA_6$,
					2/{1/2,2/5,5/1}/$\sD_6$,
					3/{2/5,3/4}/$\sE_6$,
					4/{2/3,3/4,4/5}/$\sE_6$,
					5/{1/2,1/5}/$\sD_5$
				}{
					\begin{scope}[xshift=(\k-1)*5cm]
						\foreach \i in {1,...,5}{
							\node[
								circle,draw,fill=black,
								inner sep=0.8pt,      
								minimum size=3pt     
							] (v\i) at ({360/5*(\i-1)}:1) {};
						}
						\foreach \u/\v in \edges {
							\draw (v\u) -- (v\v);
						}
						\node[below] at (0,-1.3) {\glabel};
					\end{scope}
				}
			\end{tikzpicture}
	\caption{Complete systems of representatives of $\cS_n/\sim_{\sw}$ for $n=2,3,4,5$}
			\label{fig:5}
		\end{figure}
		Each representative graph is labeled by the root lattice $\sL$ such that its Seidel matrix belongs to $\cS_n(\sL)$.
		These root lattices can be obtained by the construction given in the proof of Lemma~\ref{lem:E} with a computer.
		\begin{figure}[ht]
	  \centering
	  \begin{tikzpicture}[scale=0.5]
	    \foreach \k/\edges/\glabel in {
	      1/{1/2,2/6,6/1,3/4,4/5,5/3}/$\sA_7$,
	      2/{1/2,2/6,6/1,3/4,4/5}/$\sD_7$,
	      3/{1/2,2/6,6/1,3/5}/$\sE_7$,
	      4/{1/2,3/4,4/5,5/6}/$\sE_7$,
	      5/{1/2,6/1,4/5,3/4}/$\sD_6$
	    }{
	      \begin{scope}[xshift=(\k-1)*5cm]
	        \foreach \i in {1,...,6}{
	          \node[circle,draw,fill=black,
	                inner sep=0.8pt,minimum size=3pt]
	                (v\i) at ({360/6*(\i-1)}:1) {};
	        }
	        \foreach \u/\v in \edges {
	          \draw (v\u) -- (v\v);
	        }
	        \node[below] at (0,-1.5) {\glabel};
	      \end{scope}
	    }
	    \foreach \k/\edges/\glabel in {
	      6/{1/2,3/6,4/5}/$\sE_6$,
	      7/{2/3,3/4,4/5,5/6}/$\sE_6$,
	      8/{2/3,3/4,4/5,5/6,6/2}/$\sE_7$,
	      9/{2/3,3/5,5/6,6/2}/$\sD_5$
	    }{
	      \begin{scope}[xshift=(\k-6)*5cm,yshift=-4.5cm]
	        \foreach \i in {1,...,6}{
	          \node[circle,draw,fill=black,
	                inner sep=0.8pt,minimum size=3pt]
	                (v\i) at ({360/6*(\i-1)}:1) {};
	        }
	        \foreach \u/\v in \edges {
	          \draw (v\u) -- (v\v);
	        }
	        \node[below] at (0,-1.5) {\glabel};
	      \end{scope}
	    }
	  \end{tikzpicture}
	  \caption{A complete system of representatives of $\cS_6/\sim_{\sw}$}
	  \label{fig:6}
	\end{figure}
	Note that the two figures are consistent with Lemma~\ref{lem:S(A)}, which determines $\cS_n(\sA_m)/\sim_{\sw}$, and with Lemma~\ref{lem:S(D)}, which determines $\cS_n(\sD_m)/\sim_{\sw}$.
	\end{example}

\subsection{Construction of the sets $\cS_{n}(\sL)/\sim_{\sw}$}	
        Let $\sL$ be an irreducible root lattice, and $\br$ a root of $\sL$.
        The stabilizer subgroup $W(\sL)_\br$ of the root $\br$ in the Weyl group $W(\sL)$ acts on $N_{\br}(\sL)/\sim_\br$
        in such a way that 
        $w( [\bu]_\br ) := [w(\bu)]_\br$
        for $w \in W(\sL)_\br$ and $[\bu]_{\br}=\{\bu,\br-\bu\} \in N_{\br}(\sL)/\sim_\br$, where $\bu \in N_{\br}(\sL)$.
        Furthermore, the stabilizer subgroup $W(\sL)_\br$ naturally acts on $\binom{N_{\br}(\sL)/\sim_\br}{n}$. 
	We show that for each irreducible root lattice $\sL$, the mapping $\phi^{(n)}_{\sL, \br}$ in the following definition can construct all the switching classes in $\cS_{n}(\sL)/\sim_{\sw}$.
	Note that this mapping is well-defined by Lemma~\ref{lem:vectors-to-graphs}.
	\begin{definition}
		Let $n$ be a non-negative integer.
		Let $\sL$ be an irreducible root lattice, and $\br$ a root in $\sL$.
		Define
		\begin{align*}
			N_{\br}^{(n)}(\sL) := \left\{ \{ [\bu_1]_\br,\ldots,[\bu_n]_\br \} \in \binom{N_{\br}(\sL) / \sim_\br}{n} : \sL= \langle \bu_1,\ldots,\bu_n, \br \rangle \right\}.
		\end{align*}
		The stabilizer  subgroup $W(\sL)_\br$ of the root $\br$ in the Weyl group $W(\sL)$ naturally acts on $N_{\br}^{(n)}(\sL)$.
		Define
		\begin{align*}
			\begin{array}{rccc}
				\phi^{(n)}_{\sL, \br} : & W(\sL)_\br \backslash N_{\br}^{(n)}(\sL) & \to & \cS_{n}(\sL)/{\sim_{\sw}} \\
				& W(\sL)_\br\{ [\bu_1]_\br,\ldots, [\bu_n]_\br \} & \mapsto & [S(G)]
			\end{array}
		\end{align*}
		where $A(G)+2I = B^\top B$ with $B = [\bu_1,\ldots,\bu_n]$.
	\end{definition}

	First, we prove in the following lemma that the mapping $\phi^{(n)}_{\sL,\br}$ is surjective. 
	Consequently, an explicit determination of the domain $W(\sL)_\br \backslash N^{(n)}_{\br}(\sL)$ immediately yields the corresponding image $\cS_n(\sL)/\sim_{\sw}$.
	In Subsections~\ref{subsec:A} and~\ref{subsec:D}, we carry out this construction for root lattices of type $A$ or $D$.
	
	\begin{lemma}	\label{lem:surj}
		Let $n$ be a non-negative integer.
		Let $\sL$ be an irreducible root lattice, and $\br$ a root in $\sL$.
		Then, $\phi^{(n)}_{\sL,\br}$ is surjective.
	\end{lemma}
	\begin{proof}	
		Fix $[S] \in \cS_{n}(\sL)/{\sim_{\sw}}$.
        		By the definition of $\cS_{n}(\sL)$, there exist a graph $G$ with $S=S(G)$ satisfying $A(\hat{G})+2I = \tB^\top \tB$ for some generators $\bu_1,\ldots,\bu_n,\bu_0$ of $\sL$, where $\tB = [\bu_1,\ldots,\bu_n,\bu_0]$.
       		Here, $\bu_0$ corresponds to the apex for $G$.
		By Lemma~\ref{lem:transitive}, without loss of generality we may assume $\bu_0 = \br$.
		Then, we see that $\{ [\bu_1]_\br,\ldots,[\bu_n]_\br \} \in  N_{\br}^{(n)}(\sL)$.
		Hence, $W(\sL)_\br\{ [\bu_1]_\br,\ldots,[\bu_n]_\br \} \in W(\sL)_\br \backslash N_{\br}^{(n)}(\sL)$.
	        This means that $\phi_{\sL,\br}^{(n)}$ is surjective.
	\end{proof}

\subsection{Construction of the sets $\cS_{n}(\sA_m) /\sim_{\sw} $}	\label{subsec:A}
	In this subsection, we construct the set $\cS_{n}(\sA_m)/\sim_{\sw}$. 
	First, we determine the domain $W(\sA_m)_{\br} \backslash N_{\br}^{(n)}(\sA_m)$ of the mapping $\phi^{(n)}_{\sA_m,\br}$ for some switching root $\br$.
	\begin{lemma}	\label{lem:W(A)N}
		Let $n \geq 0$ and $m \geq 1$ be integers.
		Let $\br := \be_m - \be_{m+1} \in \sA_m$.
		Then, the set $W(\sA_m)_\br \backslash N_{\br}^{(n)}(\sA_m)$ has only one element
		$W(\sA_m)_\br\{ [ \be_i-\be_{m+1} ]_\br : i =1,\ldots,m-1 \}$ if $m=n+1$,
		and it is the empty set otherwise.
	\end{lemma}
	\begin{proof}
		Assume that we may take an element
		$
			W(\sA_m)_\br\left\{ [\bu_1]_\br ,\ldots, [\bu_n]_\br \right\} \in W(\sA_m)_\br \backslash N_{\br}^{(n)}(\sA_m).
		$	
		Since
		\begin{align} \label{lem:W(A)N:1}
			N_{\br}(\sA_m)/{\sim_\br} = \left\{  [\be_i-\be_{m+1}]_\br : i \in \{1,\ldots,m-1 \} \right\},
		\end{align}
		we have $n \leq |N_{\br}(\sA_m)/{\sim_\br}| = m-1$.
		Also, since $\sA_m = \langle \bu_1,\ldots,\bu_n,\br\rangle$ by the definition of $N_{\br}^{(n)}(\sA_m)$,
		we have $n+1 = |\{ \bu_1,\ldots,\bu_n, \br\}| \geq \rank \langle \bu_1,\ldots,\bu_n, \br \rangle = \rank(\sA_{m}) = m$.
		Thus, $n = m-1$ follows.
		This implies
		$N_{\br}^{(n)}(\sA_m) = N_{\br}^{(m-1)}(\sA_m)  = \{ N_{\br}(\sA_m)/{\sim_\br} \}$.
		This means that $W(\sA_m)_\br \backslash N_{\br}^{(n)}(\sA_m)$ has only one element $W(\sA_m)_\br (N_{\br}(\sA_m)/{\sim_\br})$.
	\end{proof}

	Next, we calculate the image $\cS_{n}(\sA_m)/\sim_{\sw}$ of the mapping $\phi^{(n)}_{\sA_m,\br}$.
	\begin{lemma}	\label{lem:S(A)}
		Let $n \geq 0$ and $m \geq 1$ be integers.
		The set $\cS_{n}(\sA_m)/ {\sim_{\sw}}$ equals $\{[S(K_n)]\}$ if $m=n+1$, and the emptyset otherwise.
	\end{lemma}
	\begin{proof}
		Let $\br := \be_m - \be_{m+1} \in \sA_m$.
		The mapping $\phi^{(n)}_{\sA_m,\br}$ is surjective by Lemma~\ref{lem:surj}, and its domain $W(\sA_m)_\br \backslash N_{\br}^{(n)}(\sA_m)$ is determined in Lemma~\ref{lem:W(A)N}.
		Hence, we calculate the image of the mapping.
		First, assume $m = n+1$. 
		Then, by Lemma~\ref{lem:W(A)N}, $\cS_{n}(\sA_m)/ {\sim_{\sw}}$ contains only one element
		\begin{align}	\label{lem:S(A):1}
			\phi^{(n)}_{\sA_m,\br} \left( W(\sA_m)_\br \{ [\be_i-\be_{m+1}]_\br : i =1,\ldots,m-1 \} \right).
		\end{align}
		Letting $B := [\be_1-\be_{m+1},\ldots,\be_{m-1}-\be_{m+1}]$, we have $A(K_n)+2I = B^\top B$.
		This shows that the element~\eqref{lem:S(A):1} is $[S(K_n)]$ as desired.
		Next, assume $m \neq n+1$.
		Then, by Lemma~\ref{lem:W(A)N}, the domain $W(\sA_m)_\br \backslash N_{\br}^{(n)}(\sA_m)$ is empty.
		Hence, the image $\cS_{n}(\sA_m)/ {\sim_{\sw}}$ is empty as well.
	\end{proof}

\subsection{Construction of the sets $\cS_{n}(\sD_m) /\sim_{\sw}$}	\label{subsec:D}
	In this subsection, we construct $\cS_{n}(\sD_m)/\sim_{\sw}$ analogously to $\cS_{n}(\sA_m)/\sim_{\sw}$. 
	Owing to the increased complexity of the argument, we begin with the following lemma, which follows from the definition of $\sD_m$.
	\begin{lemma}	\label{lem:W(D)}
		Let $m \geq 4$ be an integer.
		The Weyl group $W(\sD_m)$ contains the following.
		\begin{enumerate}
			\item All automorphisms that multiply two distinct coordinates by $-1$.	\label{lem:W(D):1}
			\item All permutations of the coordinates.	\label{lem:W(D):2}
		\end{enumerate}
	\end{lemma}
	\begin{proof}
		Let $\bu = \sum_{i=1}^m a_i \be_i \in \sD_m$.
		Let $1 \leq k \neq l \leq m$.
		Then, we have
		\begin{align*}
			s_{\be_k-\be_l} (\bu) = \left( \sum_{i \in \{1,\ldots,m\} \setminus \{k,l\}} a_i \be_i  \right) + a_k \be_l +a_l \be_k.
		\end{align*}
		Hence, the Weyl group $W(\sD_m)$ contains an automorphism which swaps the $k$th and $l$th coordinates.
		This implies~\ref{lem:W(D):2}.
		Also, we have
		\begin{align*}
			s_{\be_k + \be_l} \circ s_{\be_k-\be_l} (\bu) = \left( \sum_{i \in \{1,\ldots,m\} \setminus \{k,l\}} a_i \be_i  \right) - a_k \be_k - a_l \be_l.
		\end{align*}	
		Hence,~\ref{lem:W(D):1} holds.	
	\end{proof}

	We determine the domain $W(\sD_m)_{\br} \backslash N_{\br}^{(n)}(\sD_m)$ of the mapping $\phi^{(n)}_{\sD_m,\br}$ for some switching root $\br$.
	\begin{lemma}	\label{lem:W(D)N}
		Let $n \geq 0$ and $m \geq 4$ be integers.
		Let $\br := \be_{m-1} + \be_{m} \in \sD_m$.
		The set $W(\sD_m)_\br \backslash N_{\br}^{(n)}(\sD_m)$ has only one element
		\begin{align}	\label{lem:W(D)N:1}
			W(\sD_m)_\br\left(
				\{ [ \be_m+\be_i ]_\br : i =1,\ldots,m-2 \}
				\sqcup
				\{ [ \be_m-\be_i ]_\br : i =1,\ldots, n-m+2 \} 
			\right)
		\end{align}	
		if $2(m-2) \geq n \geq m-1$, and it is the empty set otherwise.
	\end{lemma}
	\begin{proof}
		Assume we can take an element
		$W(\sD_m)_\br\left\{ [\bu_1]_\br ,\ldots, [\bu_n]_\br \right\} \in W(\sD_m)_\br \backslash N_{\br}^{(n)}(\sD_m).$
		Then, since $\sD_m = \langle \bu_1,\ldots,\bu_n,\br\rangle$ by definition, 
		we have $n \geq m-1$ and
		\begin{align}	\label{lem:W(D)N:2}
			[\be_m+\be_i]_\br \text{ or } [\be_m-\be_i]_\br \in \{[\bu_1]_\br, \ldots, [\bu_n]_\br \}
			\qquad
			(i \in \{1,\ldots,m-2\}).
		\end{align}
		Since $n \geq m-1$,
		\begin{align*}
			N_\br(\sD_m)/\sim_{\br}
			= \{ [\be_m+\be_i ]_\br, [\be_m-\be_i]_\br : i  \in \{1,\ldots,m-2\} \}
		\end{align*}
		implies that for some $k \in \{1,\ldots,m-2\}$, both
		$
			[\be_m+\be_k]_\br
		$
		and
		$
			[\be_m-\be_k]_\br
		$
		are contained in
		$
			\{[\bu_1]_\br, \ldots, [\bu_n]_\br \}
		$.
		Next, we replace the representative $\left\{ [\bu_1]_\br ,\ldots, [\bu_n]_\br \right\}$.
		By multiplying two distinct coordinates $k$ and $i \in \{1,\ldots,m-2\} \setminus \{k\}$ by $-1$, we can change $[\be_m-\be_i]_\br$ to $[\be_m+\be_i]_\br$.
		Then, note that $[\be_m+\be_k]_\br$ and $[\be_m-\be_k]_\br$ are swapped.
		Hence, by \eqref{lem:W(D)N:2} together with Lemma~\ref{lem:W(D)}~\ref{lem:W(D):1}, we may assume
		\begin{align*}
			[\bu_1]_\br=[\be_m+\be_1]_\br, \ldots, [\bu_{m-2}]_\br = [\be_m+\be_{m-2}]_\br.
		\end{align*}
		By Lemma~\ref{lem:W(D)}~\ref{lem:W(D):2}, we may in addition assume
		\begin{align*}
			[\bu_{m-1}]_\br=[\be_m-\be_1]_\br, \ldots, [\bu_{n}]_\br = [\be_m-\be_{n-m+2}]_\br.
		\end{align*}
		Hence, the element $W(\sD_m)_\br\left\{ [\bu_1]_\br ,\ldots, [\bu_n]_\br \right\}$ coincides with the element~\eqref{lem:W(D)N:1}.
		In particular, $n-m+2 \leq m-2$.
	\end{proof}

	Next, we prepare graphs $D_{n,m}$, and determine the image $\cS_{n}(\sD_m)/\sim_{\sw}$ of the mapping $\phi^{(n)}_{\sD_m,\br}$.
	\begin{definition}
		For non-negative integers $n \geq m$, we write $D_{n,m}$ for the graph obtained from the complete graph $K_{n+m}$ of order $n+m$ by removing a matching of size $m$.
	\end{definition}
	
	\begin{lemma}	\label{lem:S(D)}
		Let $n \geq 0$ and $m \geq 4$ be integers.
		The set $\cS_{n}(\sD_m) / \sim_{\sw}$ is $\{[S(D_{m-2,n-m+2})]\}$ if $2(m-2) \geq n \geq m-1$,
		and the emptyset otherwise.
	\end{lemma}
	\begin{proof}
		In the same way as in the proof of Lemma~\ref{lem:S(A)}, we prove this lemma.
		Let $\br := \be_{m-1} + \be_{m} \in \sD_m$, and calculate the image of the mapping $\phi^{(n)}_{\sD_m,\br}$, which is surjective by Lemma~\ref{lem:surj}.
		Assume $2(m-2) \geq n \geq m-1$.
		Then, by Lemma~\ref{lem:W(D)N}, $\cS_{n}(\sD_m)/ {\sim_{\sw}}$ contains only one element
		\begin{align}	\label{lem:S(D):1}
			\phi^{(n)}_{\sD_m,\br} \left( W(\sD_m)_\br \left( \{ [ \be_m+\be_i ]_\br : i =1,\ldots,m-2 \} \sqcup \{ [ \be_m-\be_i ]_\br : i =1,\ldots, n-m+2 \}  \right) \right).
		\end{align}
		Letting $$B := [\be_m+\be_1,\ldots,\be_m+\be_{m-2},\be_{m}-\be_1,\ldots,\be_{m}-\be_{n-m+2}],$$ we have $A(D_{m-2,n-m+2})+2I = B^\top B$.
		This shows that the element~\eqref{lem:S(D):1} is $[S(D_{m-2,n-m+2})]$ as desired.
		Next, assume either $2(m-2) < n$ or $n < m-1$.
		Then, by Lemma~\ref{lem:W(D)N}, the domain $W(\sD_m)_\br \backslash N_{\br}^{(n)}(\sD_m)$ is empty.
		Hence, the image $\cS_{n}(\sD_m)/ {\sim_{\sw}}$ is empty as well.
	\end{proof}

\section{Proof of the first main result}	\label{sec:m1}
    	We prove Theorem~\ref{thm:Sn}, which is the first main result in this paper.
	Also, since $s(n) = |\cS_{n}/\sim_{\sw}|$, $t(n) = |\cT_{n}/\sim_{\sw}|$ and $\omega(n) = |\Omega_{n}/\sim_{\sw}|$, Theorem~\ref{thm:Sn} implies Corollary~\ref{cor:Sn}.
	
	\begin{theorem}	\label{thm:Sn}	
		If $n \leq 7$, then $\cS_{n}/{\sim_{\sw}} = \Omega_{n}/{\sim_{\sw}}$,
		and otherwise
		\begin{align}	\label{thm:Sn:1}
			\cS_{n}/{\sim_{\sw}} = \Omega_{n}/{\sim_{\sw}} \sqcup \{[S(K_n)]\} \sqcup \{ [S(D_{m-2,n-m+2})] : m \in \{\max\{9,\lceil n/2 \rceil+2\},\ldots,n+1\} \}.
		\end{align}
		If $n \leq 8$, then $\cT_{n}/{\sim_{\sw}} = \Theta_{n}/{\sim_{\sw}}$,
		and otherwise
		\begin{align}	\label{thm:Sn:2}
			\cT_{n}/{\sim_{\sw}} = \Omega_{n}/{\sim_{\sw}} \sqcup \{ [S(D_{m-2,n-m+2})] : m \in \{\max\{9,\lceil n/2 \rceil +2\},\ldots,n\} \}.
		\end{align}				
	\end{theorem}
	\begin{proof}
		First we prove $\cS_{n}/{\sim_{\sw}} = \Omega_{n}/{\sim_{\sw}}$ if $n \leq 7$.
		It suffices to prove $\cS_{n} \subset \Omega_{n}$.
		Indeed,  we fix $S \in \cS_{n}$.
		Then by $n \leq 7$, we have $\rank(3I-S) \leq 7$.
		Hence, we have $S \in \Omega_{n}$.
		
		Next we assume $n \geq 8$, and show~\eqref{thm:Sn:1}.
		By Lemmas~\ref{lem:S(A)} and~\ref{lem:S(D)}, we have
		\begin{align*}
			\bigsqcup_{\sL \in \cL \setminus \cE } \cS_{n}(\sL)/\sim_{\sw}
			&=\cS_{n}(\sA_{n+1})/\sim_{\sw} \sqcup \left( \bigsqcup_{m \in \{\max\{9,\lceil n/2 \rceil+2\},\ldots,n+1\}} \cS_{n}(\sD_m)/{\sim_{\sw}} \right) \\
			&=\{ [S(K_n)] \} \sqcup \left\{ [S(D_{m-2,n-m+2})] : m \in \{\max\{9,\lceil n/2 \rceil+2\},\ldots,n+1\}  \right\}.
		\end{align*}
		Since $\cS_{n}/{\sim_{\sw}}$ is decomposed as \eqref{lem:Seidel decomp2:1} by Lemma~\ref{lem:Seidel decomp2},
		we obtain~\eqref{thm:Sn:1}.
		
		By~\eqref{lem:Seidel decomp2:2} in~Lemma~\ref{lem:Seidel decomp2}, we have $\cT_{n}/\sim_{\sw} = \Theta_{n}/\sim_{\sw}$ if $n \leq 8$.
		Next, we assume $n \geq 9$ and show~\eqref{thm:Sn:2}.
		Similarly, we have by Lemmas~\ref{lem:S(A)} and~\ref{lem:S(D)}
		\begin{align*}
			\bigsqcup_{\sL \in \cL \setminus \cE \text{ with } \rank(\sL) \leq n} \cS_{n}(\sL)/{\sim_{\sw} } 
			&= \bigsqcup_{m \in \{\max\{9,\lceil n/2 \rceil +2\},\ldots,n\}} \cS_{n}(\sD_m)/{\sim_{\sw}} \\
			&= \left\{ [S(D_{m-2,n-m+2})] : m \in \{\max\{9,\lceil n/2 \rceil +2\},\ldots,n\}  \right\}.
		\end{align*}
		Since $\Theta_{n} = \Omega_{n}$ and the set $\cT_{n}/\sim_{\sw}$ is decomposed as in~\eqref{lem:Seidel decomp2:2} by Lemma~\ref{lem:Seidel decomp2},
		we have~\eqref{thm:Sn:2}.
	\end{proof}

\section{Isometries}	\label{sec:embeddings}
	The remainder of this paper is devoted to proving the second main result, Theorem~\ref{thm:sym}, 
	concerning Seidel matrices $S$ with maximum eigenvalue $3$ and $\rank(3I-S) \le 7$, 
	or equivalently, sets of equiangular lines with common angle $\arccos(1/3)$ in dimension $7$.  
	Such Seidel matrices arise from roots in root lattices of rank at most $8$.
	As these lattices admit embeddings into $\sE_8$,
	we present several lemmas on these embeddings in this section.

    \begin{definition}	\label{dfn:hom}
        Let $\sL$ and $\sM$ be root lattices.
        Denote by $\Hom(\sL,\sM)$ the set of isometries from $\sL$ to $\sM$.
        For $\bv \in \sL$ and $\bu \in \sM$, define
        $
            \Hom((\sL,\bv),(\sM,\bu)) := \{ f \in \Hom(\sL,\sM)  :  f(\bv)=\bu \}.
        $
    \end{definition}
        Note that the automorphism group $\Aut(\sM)$ naturally acts on $\Hom(\sL,\sM)$
        in such a way that
        $
            g_*( f ) := g \circ f
        $
        for $g \in \Aut(\sM)$ and $f \in \Hom(\sL,\sM)$.
    By Lemma~\ref{lem:transitive}, we have the following lemma.
    \begin{lemma} \label{lem:base}
        Let $\sL$ and $\sM$ be irreducible root lattices.
        For two roots $\bv \in \sL$ and $\bu \in \sM$, the mapping
        \begin{align*}
            \begin{array}{ccc}
                W(\sM)_\bu \backslash \Hom((\sL,\bv),(\sM,\bu)) & \to & W(\sM) \backslash \Hom(\sL,\sM) \\
                W(\sM)_\bu f & \mapsto & W(\sM) f
            \end{array}
        \end{align*}
        is bijective.
    \end{lemma}

    We have the following theorem by~\cite[Theorem~3.5]{Oshima2007}.
    \begin{theorem}	\label{thm:Oshima}
	Let $\sL$ be an irreducible root lattice of rank at most $8$.
	The cardinality of $W(\sE_8) \backslash \Hom(\sL,\sE_8)$ equals $2$ if $\sL \in \{\sA_7, \sD_8\}$,
        and $1$ otherwise.
    \end{theorem}

\subsection{Isometries from $\sA_7$ to $\sE_8$}
	As stated in Theorem~\ref{thm:Oshima}, the root lattices $\sA_7$ and $\sD_8$ admit two distinct isometries to $\sE_8$ up to the action of $W(\sE_8)$.  
	Consequently, we need additional lemmas for these two cases.  
	In this subsection, our goal is to prove Lemma~\ref{lem:subA7E8}, which concerns the isometries from the lattice $\sA_7$ to $\sE_8$.

	\begin{definition}
		Let $\sA_7' := \langle \{-\be_1-\be_2\} \cup \{\be_i-\be_{i+1} : i=2,\ldots,7\} \rangle$.
	\end{definition}
	Since the lattice $\sA_7$ has a basis $\{\be_i-\be_{i+1} : i=1,\ldots,7\}$, 
	we see that $\sA_7'$ and $\sA_7$ are isometric.
	\begin{lemma}	\label{lem:A}
		We have $W(\sE_8)\sA_7 \neq W(\sE_8)\sA_7'$.
	\end{lemma}
	\begin{proof}
		We see that $\sA_7^\perp \cap \sE_8 = \langle \bj/2 \rangle$ and $\sA_7'^\perp \cap \sE_8 = \langle \bj-2\be_1 \rangle$ hold.
		Since the norms of $\bj/2$ and $\bj-2\be_1$ are $2$ and $8$, respectively, we obtain $W(\sE_8)\sA_7 \neq W(\sE_8)\sA_7'$.
	\end{proof}
    	
	With these preparations, we now proceed to prove the desired lemma.
	\begin{lemma}	\label{lem:subA7E8}
		Let $\br = \be_7-\be_8$. Then, 
		\begin{align*}
			\{ W(\sE_8)_\br f(\sA_7)  :  f \in \Hom((\sA_7,\br),(\sE_8,\br)) \}= \{ W(\sE_8)_\br \sA_7, W(\sE_8)_\br \sA_7' \}.
		\end{align*}	
	\end{lemma}
	\begin{proof}
		In this proof, we write $\iota$ for the inclusion from $\sA_7$ to $\sE_8$.
		Also, we let $\rho : \sA_7 \to \sE_8$ be an isometry satisfying
		\begin{align*}
			\rho(\be_1-\be_2) = -\be_1-\be_2
			\qquad \text{ and }\qquad
			\rho(\be_i-\be_{i+1}) = \be_i-\be_{i+1} \quad (i=2,\ldots,7).
		\end{align*}
		Then, $\iota$ and $\rho$ are in $\Hom((\sA_7,\br),(\sE_8,\br))$. 
		We have 
		\begin{align*}
			\{ W(\sE_8)_\br f(\sA_7)  :  f \in \Hom((\sA_7,\br),(\sE_8,\br)) \} 
			&\supset \{ W(\sE_8)_\br \iota(\sA_7), W(\sE_8)_\br \rho(\sA_7) \} \\
			&\supset \{ W(\sE_8)_\br \sA_7, W(\sE_8)_\br \sA_7' \}.
		\end{align*}
		By Lemma~\ref{lem:A}, we have
		\begin{align*}
			|\{ W(\sE_8)_\br \sA_7, W(\sE_8)_\br \sA_7' \}|=2.
		\end{align*}
		Hence, by Lemma~\ref{lem:base} and Theorem~\ref{thm:Oshima}, we have
		\begin{align*}
			| \{ W(\sE_8)_\br f(\sA_7)  :  f \in \Hom((\sA_7,\br),(\sE_8,\br)) \} | = 2.
		\end{align*}
		Therefore, the desired result follows.
	\end{proof}

\subsection{Isometries from $\sD_8$ to $\sE_8$}
    In this subsection, our goal is to prove Lemma~\ref{lem:subD8E8}, which concerns the isometries from the lattice $\sD_8$ into $\sE_8$.
    We begin by defining the following mappings.
    \begin{definition}
		Let $\iota$ be the inclusion from $\sD_8$ to $\sE_8$.	
		For each $k \in \{1,\ldots, 8 \}$, we define $\rho_k : \sD_8 \to \sD_8$ as
		\begin{align*}
			\rho_k \left( \sum_{i =1}^8 a_i\be_i \right) := -a_k \be_k + \sum_{i \in \{1,\ldots,8 \} \setminus \{k\}} a_i\be_i.
		\end{align*}
    \end{definition}	 
    Next, we explicitly describe the two isometries of $\sD_8$ into $\sE_8$ up to the action of $W(\sD_8)$.
    \begin{lemma} \label{lem:D}
		For each $k \in \{1,\ldots, 8 \}$, 
		$
            W(\sE_8) \backslash \Hom(\sD_8,\sE_8) = \{ W(\sE_8) \iota,  W(\sE_8)( \iota \circ \rho_k ) \}.
		$
    \end{lemma}
	\begin{proof}
		Since Theorem~\ref{thm:Oshima} asserts $|W(\sE_8) \backslash \Hom(\sD_8,\sE_8)|=2$, it suffices to verify 
		\begin{align}	\label{ex:D:1}
			W(\sE_8) \iota \neq W(\sE_8)( \iota \circ \rho_k ).
		\end{align}
			By way of contradiction, we assume that there exists $w \in W(\sE_8)$ such that $w \circ \iota = \iota \circ \rho_k$.
		Since $\rank \sE_8 = \rank \sD_8$, we see that if $\sum_{i =1}^8 a_i\be_i \in \sE_8$, then
		\begin{align*}
			w \left( \sum_{i =1}^8 a_i\be_i \right) = -a_k \be_k + \sum_{i \in \{1,\ldots,8 \} \setminus \{k\}} a_i\be_i.
		\end{align*}
			Although $\bj/2 \in \sE_8$, we obtain $w ( \bj/2 ) = \bj/2 - \be_k \not\in \sE_8$.
		This is a contradiction, and hence~\eqref{ex:D:1} follows.
	\end{proof}
	
	From Lemmas~\ref{lem:base} and~\ref{lem:D}, we obtain the following lemma.
	\begin{lemma}	\label{lem:subD8E8}
		Let $\br = \be_7+\be_8$. 
		Then,
        $
            W(\sE_8)_{\br} \backslash \Hom((\sD_8,\br),(\sE_8,\br)) = \{W(\sE_8)_{\br} \iota, W(\sE_8)_{\br}(\iota \circ \rho_k)\}
        $
		for $k=1,\ldots,6$.
	\end{lemma}

\section{Proof of the second main result}
    \label{sec:relations}
    In this section, we prove the second main result, Theorem~\ref{thm:sym} and Corollary~\ref{cor:sym} by using the mapping $\phi_{\br}^{(n)}$ in the following definition. 

    \begin{definition}
        Let $n$ be a non-negative integer, and $\br$ be a root of $\sE_8$.
        We define
        \begin{align*}
            \begin{array}{rccc}
                \phi_{\br}^{(n)} : & W(\sE_8)_\br \backslash \binom{N_{\br}(\sE_8)/\sim_\br}{n} & \to & \Omega_{n}/{\sim_{\sw}} \\
                & W(\sE_8)_\br\{ [\bu_1]_\br,\ldots, [\bu_n]_\br \} & \mapsto & [S(G)]
            \end{array}
        \end{align*}
        where $A(G)+2I = B^\top B$ with $B = [\bu_1,\ldots,\bu_n]$.
        For short, we set $\phi^{(n)} := \phi_{\br}^{(n)}$.
    \end{definition}
    By Theorem~\ref{thm:Cao} and Lemma~\ref{lem:vectors-to-graphs}, we see that the mapping $\phi^{(n)}$ is well-defined.
    First, we provide two lemmas to determine some fibre of the mapping $\phi^{(n)}$.
    Also, in these lemmas, note that $f( \{ [\bv_1]_{\bv_0},\ldots, [\bv_n]_{\bv_0} \} )$ is naturally defined by $\{ [f(\bv_1)]_{f(\bv_0)},\ldots, [f(\bv_n)]_{f(\bv_0)} \}$.

    \begin{lemma}	\label{lem:A7E8}
    	Let $\br$ be a root of $\sE_8$.
    	Let $\bv_1,\ldots,\bv_n,\bv_0$ be roots.
	Assume that there exists a graph $G$ such that $A(\hat{G})+2I = \tB^\top \tB$, where $\tB = [\bv_1,\ldots,\bv_n,\bv_0]$.
    	Here, the vector $\bv_0$ corresponds to the apex for $G$.
    	If $\sL := \langle \bv_1,\ldots,\bv_n,\bv_0 \rangle$ is isometric to $\sA_7$,
    	then 	$G$ is switching equivalent to $K_6$, and
	\begin{align}	\label{lem:A7E8:1}
	    \left|  \left\{
                W(\sE_8)_\br f( \{ [\bv_1]_{\bv_0},\ldots, [\bv_n]_{\bv_0} \} ) 
                 : 
                f \in \Hom((\sL,\bv_0),(\sE_8,\br))
            \right\} \right| = 2.
        \end{align}
    \end{lemma}

    \begin{proof}
        Without loss of generality, we may assume that $\sL = \sA_7 \subset \sE_8$ and $\bv_0 = \br=\be_7-\be_8$ by Lemma~\ref{lem:transitive}.
        Set $$\cV := \{ [\bv_1]_{\bv_0},\ldots, [\bv_n]_{\bv_0} \}$$ for short.
    	Since $\bv_1,\ldots,\bv_n,\br=\bv_0$ generates $\sA_7$, we see that $\cV \in N_\br^{(n)}(\sA_7)$.
	By Lemma~\ref{lem:W(A)N}, we have $n=6$ and take $g \in W(\sA_7)_\br$ such that
	\begin{align*}
		g(\cV) = \{ [\be_1-\be_8]_\br,\ldots,[\be_6-\be_8]_\br \}.
	\end{align*}
	Hence, $G$ is switching equivalent to $K_6$.

	Next we prove~\eqref{lem:A7E8:1}.
        Theorem~\ref{thm:Oshima} together with Lemma~\ref{lem:base} implies
        \begin{align*}
            \left|  \left\{
                W(\sE_8)_\br f( \cV ) 
                 : 
                f \in \Hom((\sL,\bv_0),(\sE_8,\br))
            \right\} \right| 
            &\leq 
            |W(\sE_8)_\br \backslash \Hom((\sL,\bv_0),(\sE_8,\br))| \\
            &=
            |W(\sE_8) \backslash \Hom(\sL,\sE_8)| 
            =2.
        \end{align*}
        Also, by Lemma~\ref{lem:subA7E8},
        we have
        \begin{align*}
            \left|  \left\{
                W(\sE_8)_\br f( \cV ) 
                 : 
                f \in \Hom((\sL,\bv_0),(\sE_8,\br))
            \right\} \right|
            \geq
            \left|  \left\{
                W(\sE_8)_\br f( \sL ) 
                 : 
                f \in \Hom((\sL,\bv_0),(\sE_8,\br))
            \right\} \right|
            = 2.
        \end{align*}
        These inequalities imply the desired result~\eqref{lem:A7E8:1}.
    \end{proof}

    \begin{lemma}	\label{lem:D8E8}
        Let $\br$ be a root of $\sE_8$.
       Let $\bv_1,\ldots,\bv_n,\bv_0$ be roots.
	Assume that there exists a graph $G$ such that $A(\hat{G})+2I = \tB^\top \tB$, where $\tB = [\bv_1,\ldots,\bv_n,\bv_0].$
        Here, the vector $\bv_0$ corresponds to the apex for $G$.
        If $\sL := \langle \bv_1,\ldots,\bv_n,\bv_0 \rangle$ is isometric to $\sD_8$,
        then
        \begin{align*}
            \left|  \left\{
                W(\sE_8)_\br f( \{ [\bv_1]_{\bv_0},\ldots, [\bv_n]_{\bv_0} \} ) 
                 : 
                f \in \Hom((\sL,\bv_0),(\sE_8,\br))
            \right\} \right| = 1.
        \end{align*}
    \end{lemma}
    \begin{proof}
        Without loss of generality, we may assume that $\sL=\sD_8$ and $\br = \bv_0 :=\be_7+\be_8$ by Lemma~\ref{lem:transitive}.
        Setting
        $\cV := \{ [\bv_1]_{\bv_0},\ldots, [\bv_n]_{\bv_0} \}$,
        we consider 
        \begin{align}	\label{lem:D8E8:0}
            \left\{
                W(\sE_8)_\br f( \cV ) 
                 : 
                f \in \Hom((\sL,\bv_0),(\sE_8,\br))
            \right\}.
        \end{align}
        Since $\cV \in N_{\br}^{(n)}(\sD_8)$, we have $7 \leq n \leq 12$ by Lemma~\ref{lem:W(D)N}.
	Set
	$
        	      \cV' := \{ [ \be_8+\be_i ]_\br : i \in \{1,\ldots,6\}\} \sqcup \{ [ \be_8-\be_i ]_\br : i \in \{1,\ldots,n-6\} \}.
        $
        Since $\cV' \in N_{\br}^{(n)}(\sD_8)$, we see by Lemma~\ref{lem:W(D)N} that $W(\sD_8)_{\br} \cV = W(\sD_8)_{\br} \cV'$.
        Hence, \eqref{lem:D8E8:0} equals
       \begin{align}	\label{lem:D8E8:1}
            \left\{ W(\sE_8)_\br f( \cV')  :  f \in \Hom((\sL,\bv_0),(\sE_8,\br))  \right\}.
        \end{align}
        By Lemma~\ref{lem:subD8E8} with $k=1$,
        \eqref{lem:D8E8:1} equals
        	\begin{align}	\label{lem:D8E8:2}
		\left\{
        		W(\sE_8)_\br \cV'  ,
                 W(\sE_8)_\br \rho_1 ( \cV' )
                 \right\}.
        \end{align}
        Since $\rho_1([\be_8+\be_1]_\br) = [\be_8-\be_1]_\br$ and $\rho_1([\be_8-\be_1]_\br)=[\be_8+\be_1]_\br$,
        we have $W(\sE_8)_\br \cV' = W(\sE_8)_\br \rho_1 ( \cV' )$.
        Hence, the set~\eqref{lem:D8E8:2} has exactly one element.
        This is the desired conclusion.
    \end{proof}
    
    The following is the second main result.
    \begin{theorem}	\label{thm:sym}
        The following hold.
        \begin{enumerate}
            \item For a non-negative integer $n \ne 6$, the mapping $\phi^{(n)}$ is bijective.
            \item The fibre of $[S(K_6)]$ under $\phi^{(6)}$ has exactly two elements,
        and the other fibres have only one element.
        \end{enumerate}
    \end{theorem}
    \begin{proof}
        Fix $[S] \in \Omega_{n}/{\sim_{\sw}}$.
        We consider the fibre of $[S]$ under $\phi^{(n)}$.
        By Lemma~\ref{lem:E} together with $\rank(3I-S) \leq 7$, we have $S \in \cS_{n}(\sL)$ for some $\sL \in \cL$ with $\rank(\sL) \leq 8$.
       Hence, we can take a graph $G$ with $S=S(G)$ satisfying $A(\hat{G})+2I = \tB^\top \tB$ for generators $\bv_1,\ldots,\bv_n,\bv_0$ of $\sL$, where $\tB = [\bv_1,\ldots,\bv_n, \bv_0 ]$.
        Here, $\bv_0$ corresponds to the apex for $G$.
%
        The fibre $(\phi^{(n)})^{-1}([S(G)])$ is
        \begin{align*}
            \left\{ W(\sE_8)_\br \{ [\bu_1]_\br,\ldots, [\bu_n]_\br \} \in W(\sE_8)_\br \backslash \binom{N_{\br}(\sE_8)/\sim_\br}{n} :
            \phi^{(n)}(W(\sE_8)_\br\{ [\bu_1]_\br,\ldots, [\bu_n]_\br \}) = [S]
            \right\}.
        \end{align*}
        This can be written as
        \begin{align*}
            \left\{ W(\sE_8)_\br f( \{ [\bv_1]_{\bv_0},\ldots, [\bv_n]_{\bv_0} \} )  : 
            f \in \Hom((\sL,\bv_0),(\sE_8,\br)) 
            \right\}.
        \end{align*}
        Hence, if $\sL \neq \sA_7, \sD_8$, then this fibre has only one element since
        \begin{align*}
            |W(\sE_8)_\br \backslash \Hom((\sL,\bv_0),(\sE_8,\br)) | = 1
        \end{align*}
        by Lemma~\ref{lem:base} and Theorem~\ref{thm:Oshima}.
        Next if $\sL$ is isometric to $\sD_8$, then the fibre has only one element by Lemma~\ref{lem:D8E8}.
        If $\sL$ is isometric to $\sA_7$, then $G$ is switching equivalent to $K_6$ and the fibre has two elements by Lemma~\ref{lem:A7E8}.
    \end{proof}

	Finally, we prove Corollary~\ref{cor:sym}.  
	This proof shows that the almost symmetry of the numbers $\omega(n)$ arises from the operation of taking complements in some set.
    \begin{proof}[Proof of Corollary~\ref{cor:sym}]
    	Recall that $\omega(n)=|\Omega_{n}/{\sim_{\sw}}|$.
	Let $\br$ be a root of $\sE_8$.
    	For a non-negative integer $n$ at most $28$, the mapping
    	\begin{align*}
    	    \begin{array}{cccc}
    	        c^{(n)} : &\binom{N_{\br}(\sE_8)/\sim_\br}{n} & \to & \binom{N_{\br}(\sE_8)/\sim_\br}{28-n} \\
    	        &X & \mapsto & (N_{\br}(\sE_8)/\sim_\br) \setminus X
    	    \end{array}
    	\end{align*}
    	is bijective.
	The mapping $c^{(n)}$ induces the bijective mapping
    	\begin{align*}
    	    \begin{array}{cccc}
    	        \bar{c}^{(n)}:&W(\sE_8)_\br \backslash  \binom{N_{\br}(\sE_8)/\sim_\br}{n} & \to &W(\sE_8)_\br \backslash  \binom{N_{\br}(\sE_8)/\sim_\br}{28-n} \\
	        & W(\sE_8)_\br X & \mapsto & W(\sE_8)_\br \left( (N_{\br}(\sE_8)/\sim_\br) \setminus X \right)
    	    \end{array}.
    	\end{align*}
	Fix $n \in \{ 0,\ldots,14\}$.
	Since $\phi^{(28-n)}$ is bijective by Theorem~\ref{thm:sym},
	we obtain the composite mapping
    	\begin{align*} 
		 \Omega_{28-n}/{\sim_{\sw}}
		 \xrightarrow{(\phi^{(28-n)})^{-1}}
		 W(\sE_8)_\br \backslash \binom{N_{\br}(\sE_8)/\sim_\br}{28-n}
		 \xrightarrow{\bar{c}^{(28-n)}}  W(\sE_8)_\br \backslash \binom{N_{\br}(\sE_8)/\sim_\br}{n}
		 \xrightarrow{\phi^{(n)}}
		 \Omega_{n}/{\sim_{\sw}}.
	\end{align*}
	If $n \neq 6$, then this mapping is bijective by Theorem~\ref{thm:sym},
	and hence $|\Omega_{28-n}/{\sim_{\sw}}|=|\Omega_{n}/{\sim_{\sw}}|$.
	Also, in the case of $n=6$, we see by Theorem~\ref{thm:sym} that the fibre of $[S(K_6)]$ under the mapping $\phi^{(n)} \circ \bar{c}^{(28-n)} \circ (\phi^{(28-n)})^{-1}$ has exactly two elements,
        and the other fibres have only one element.
	Hence, $|\Omega_{22}/{\sim_{\sw}}|=|\Omega_{6}/{\sim_{\sw}}|+1$ follows as desired.	
    \end{proof}

\section*{Acknowledgements}
\indent
The second main result is part of my doctoral dissertation.  
I am grateful to my advisor, Professor Munemasa, for his invaluable guidance and support.  
I also thank Masaaki Harada, Hiroki Shimakura, and Hajime Tanaka for their helpful comments.  
\bibliographystyle{plain}
\bibliography{references}

\end{document}